\documentclass[11pt]{amsart}

\usepackage{amsmath}
\usepackage{amssymb}
\usepackage{amsthm}
\usepackage{wrapfig}
\usepackage[dvipdfm]{graphicx}

\newtheorem{definition}{\bf Definition}[section]
\newtheorem{theorem}[definition]{\bf Theorem}
\newtheorem{lemma}[definition]{\bf Lemma}
\newtheorem{proposition}[definition]{\bf Proposition}

\numberwithin{equation}{section}

\title[The Mosco convergence of Dirichlet forms on thin domains]
{The Mosco convergence of Dirichlet forms approximating the Laplace operators with the delta potential on thin domains}
\author{Hirotoshi Kuroda}
\address{Mathematical Institute, Graduate School of Science, Tohoku University, 
6-3, Aoba, Aramaki, Aoba-ku, Sendai 980-8578, Japan}
\email{h-kuroda@math.tohoku.ac.jp}

\begin{document}
\footnote[0]{This research was supported by JST, CREST:
A Mathematical Challenge to a New Phase of Material Science, Based on Discrete Geometric Analysis.}
\begin{abstract}
We consider the convergent problems of Dirichlet forms associated with the Laplace operators on thin domains.
This problem appears in the field of quantum waveguides.
We study that a sequence of Dirichlet forms approximating the Laplace operators with the delta potential on thin domains Mosco converges
to the form associated with the Laplace operator with the delta potential on the graph in the sense of Gromov-Hausdorff topology.
From this results we can make use of many results established by Kuwae and Shioya about the convergence of the semigroups
and resolvents generated by the infinitesimal generators associated with the Dirichlet forms.
\end{abstract}
\maketitle

\section{Introduction}

In this paper we consider the Mosco convergence of energy functionals associated with Laplace operators in thin domains
with respect to the Gromov-Hausdorff topology.
Roughly speaking, 
we define thin domains $\Omega_{\varepsilon}$, which sometimes called fat graphs, a simple graph $G = (V,E)$ 
and take some suitable weighted measure $d\mu_{\varepsilon}$ on $\Omega_{\varepsilon}$. 
Here we call the simple graph if the graph has only one junction point (see Figure 1).
Moreover we define the energy functional $\varphi_{\varepsilon}$ associated with the Neumann Laplacian 
on the thin domain $\Omega_{\varepsilon}$.
These notations are defined in section 3.1 and 3.2.
Our main topic is to prove that the energy functional sequence $\{ \varphi_{\varepsilon} \}$ Mosco converges to the functional $\varphi$
associated with Kirchhoff type Laplacian on the graph in the sense of Gromov-Hausdorff topology.
These definitions of functionals $\varphi_{\varepsilon}$ and $\varphi$ are given in section 3.3.

Many mathematicians studied quantum waveguides under various boundary conditions.
First, Hale and Raugel \cite{Hale-Raugel} considered the squeezing problems in thin domains.
Recently, Bouchitte et al. \cite{BMT} showed the Mosco convergence of 
the energy functional sequence associated with Dirichlet Laplacian on thin tubular domains 
in the case which the graph is only one finite space curve.
Kosugi \cite{Kosugi} considered the convergence of solutions to elliptic equations with Neumann boundary conditions on thin domains. 
Kuwae and Shioya \cite{Kuwae-Shioya} proved that the Mosco convergence of the quadratic forms implies
the convergence of the semigroups and the resolvents of self-adjoint operators associated with these forms.
So our results are more convenient to apply to the problem in the quantum physics.
We refer the papers \cite{Dal, Mosco} for more details about $\Gamma$-convergence, Mosco convergence and their applications.
Also many researchers considered the Dirichlet boundary problems in thin domains\cite{ACD, Exner-Post, Molchanov-Vainberg}.
In this Dirichlet problems more complicated structures appear in the gluing conditions at the vertices on the limiting graph.

We explain the table of contents in this paper.
In section 2 we recall ideas about the Gromov-Hausdorff topology established by Kuwae and Shioya \cite{Kuwae-Shioya}.
In section 3 we prove that a sequence of thin domains converges to a graph with respect to the Gromov-Hausdorff topology
and define energy functionals on the space of all square integrable functions on thin domains.
In section 4 we prove that our energy functional sequence satisfies asymptotic compactness condition.
To obtain the continuity of the limit function on the graph is most difficult part in this story.
In section 5 we consider our main theorem that 
Dirichlet forms, which associated with the Neumann Laplace operators added the potential function approximating the delta function,
on thin domains Mosco converges to Dirichlet form, which associated with the delta type Laplace operator.
In this paper we treat a simple graph to avoid the complicated notations which may disturb to understand essential ideas.
For that reason, we remark that our method works about more general network shaped graph in last section 6.

\section{Preliminaries}

We recall some definitions and propositions about Gromov-Hausdorff topology introduced by Kuwae and Shioya \cite{Kuwae-Shioya}.
We refer to that paper for more details.

First, denote by $\mathcal{M}_c$ the set of isomorphism classes of triples $(X,p,m)$, 
where $X$ is a locally compact separable metric space
such that any bounded subset of $X$ is relatively compact, $p \in X$ and $m$ is a positive Radon measure on $X$.
Let $\mathcal{A}$ and $\mathcal{B}$ be any directed sets.

\begin{definition}[cf. {\cite[Remark 2.2]{Kuwae-Shioya}}]\label{d:mcGH}$($measured Gromov-Hausdorff convergence$)$
	We say that a net $\{ (X_{\alpha},p_{\alpha},m_{\alpha}) \}_{\alpha \in \mathcal{A}}$ of spaces in $\mathcal{M}_c$ converges to 
	a space $(X,p,m) \in \mathcal{M}_c$ in the sense of pointed, measured, and compact Gromov-Hausdorff convergence 
	if and only if there exist nets of positive numbers 
	$\{ r_{\alpha} \}_{\alpha \in \mathcal{A}}, \, \{ r_{\alpha}^{\prime} \}_{\alpha \in \mathcal{A}}$ and
	$\{ \varepsilon_{\alpha} \}_{\alpha \in \mathcal{A}}$ such that $r_{\alpha}, r_{\alpha}^{\prime} \nearrow \infty$,
	$\varepsilon_{\alpha} \searrow 0$, and
	$m_{\alpha}$-measurable maps $f_{\alpha} : B(p_{\alpha}, r_{\alpha}) \rightarrow B(p,r_{\alpha}^{\prime})$ 
	called $\varepsilon_{\alpha}$-approximations for $\alpha \in \mathcal{A}$, such that
	\[ | d(f_{\alpha}(x),f_{\alpha}(y)) - d_{\alpha}(x,y)| < \varepsilon_{\alpha} \qquad 
		\mathrm{for \ any} \ x,y \in B(p_{\alpha},r_{\alpha}), \, \alpha \in \mathcal{A}, \]
	\[ B(p,r_{\alpha}^{\prime}) \subset B(f_{\alpha}(B(p_{\alpha},r_{\alpha})),\varepsilon_{\alpha}) \qquad 
		\mathrm{for} \ \alpha \in \mathcal{A}, \]
	\[ \lim_{\alpha} \, \int_{B(p_{\alpha},r_{\alpha})} u \circ f_{\alpha} \, dm_{\alpha} = \int_{X} u \, dm \qquad
		\mathrm{for \ any} \ u \in C_0(X), \]
	where $d_{\alpha}, \, d$ denote the distance functions on $X_{\alpha}, \, X$,
	$B(A,r)$ the open metric $r$-ball of a set $A$ in a metric space,
	and $C_0(X)$ the set of real valued continuous functions on $X$ with compact support in $X$.
\end{definition}

\medskip

\begin{definition}[cf. {\cite[Definition 2.3]{Kuwae-Shioya}}]\label{d:st-c}$(L^2$-strong convergence$)$
	Let $\{ (X_{\alpha},p_{\alpha},m_{\alpha}) \}_{\alpha \in \mathcal{A}}$ be a \\ net of spaces in $\mathcal{M}_c$ 
	and $(X,p,m) \in \mathcal{M}_c$ be a space.
	A net $\{ u_{\alpha} \}_{\alpha \in \mathcal{A}}$ with $u_{\alpha} \in L^2(X_{\alpha},m_{\alpha})$ is said to 
	strongly $L^2$-converges to a function $u \in L^2(X,m)$ if $\{ (X_{\alpha},p_{\alpha},m_{\alpha}) \}_{\alpha \in \mathcal{A}}$
	converges to $(X,p,m)$ with respect to the pointed, measured, and compact Gromov-Hausdorff topology and 
	if there exists a net $\{ \tilde{u}_{\beta} \}_{\beta \in \mathcal{B}}$ of functions in $C_0(\mathrm{supp} \, m)$ 
	tending to $u$ in $L^2(X,m)$ such that
	\[ \lim_{\beta} \, \limsup_{\alpha} \, \| \Phi_{\alpha} \tilde{u}_{\beta} - u_{\alpha} \|_{L^2(X_{\alpha},m_{\alpha})} = 0, \]
	where $f_{\alpha} : B(p_{\alpha}, r_{\alpha}) \rightarrow B(p, r_{\alpha}^{\prime})$ are $\varepsilon_{\alpha}$-approximations, and
	for $v \in C_0(\mathrm{supp} \, m)$,
	\[ \Phi_{\alpha} v := \left\{ \begin{array}{cl}
		v \circ f_{\alpha} & \quad \mathrm{on} \ \ B(p_{\alpha}, r_{\alpha}), \\[0.2cm]
		0 & \quad \mathrm{on} \ \ X_{\alpha} \setminus B(p_{\alpha}, r_{\alpha}).
		\end{array} \right. \]
\end{definition}

\medskip

We define the Hilbert spaces $H_{\alpha} := L^2(X_{\alpha},m_{\alpha})$ and $H := L^2(X,m)$ and
assume that $\{ (X_{\alpha},p_{\alpha},m_{\alpha}) \}_{\alpha \in \mathcal{A}}$ converges to $(X,p,m)$ 
in the sense of Definition \ref{d:mcGH}. 
We remark that we can define the following concepts for general Hilbert spaces,
but in this paper we need only $L^2$ spaces.
More general definitions are treated in \cite{Kuwae-Shioya}. 

\begin{definition}[cf. {\cite[Definition 2.5]{Kuwae-Shioya}}]\label{d:we-c}$(L^2$-weak convergence$)$
	We say that a net $\{ u_{\alpha} \}_{\alpha \in \mathcal{A}}$ with $u_{\alpha} \in H_{\alpha}$ weakly converges to
	a function $u \in H$ if 
	\[ \lim_{\alpha} \, \langle u_{\alpha}, v_{\alpha} \rangle_{H_{\alpha}} = \langle u,v \rangle_{H} \]
	for any net $\{ v_{\alpha} \}_{\alpha \in \mathcal{A}}$ with $v_{\alpha} \in H_{\alpha}$ tending strongly to $v \in H$.
\end{definition}

\bigskip

Next, we recall properties about quadratic forms on Hilbert spaces.
\begin{definition}[cf. {\cite[Definition 2.8]{Kuwae-Shioya}}]\label{d:g-c}$(\Gamma$-convergence$)$ \ 
	We say that a net 
	$\{ F_{\alpha} : H_{\alpha} \rightarrow \overline{\mathbb{R}} := \mathbb{R} \cup \{ \pm \infty \} \}_{\alpha \in \mathcal{A}}$
	of functions $\Gamma$-converges to a function $F : H \rightarrow \overline{\mathbb{R}}$ 
	if and only if $(\mathrm{F}1)$ and $(\mathrm{F}2)$ hold:
	\begin{enumerate}
	\item[$(\mathrm{F}1)$]
			If a net $\{ u_{\alpha} \}_{\alpha \in \mathcal{A}}$ with $u_{\alpha} \in H_{\alpha}$ strongly converges to $u \in H$, then
				\[ F(u) \leq \liminf_{\alpha} \, F_{\alpha}(u_{\alpha}). \]
	\item[$(\mathrm{F}2)$]
			For any $u \in H$ there exists a net $\{ u_{\alpha} \}_{\alpha \in \mathcal{A}}$ with $u_{\alpha} \in H_{\alpha}$ 
			which strongly converges to $u$ and
					\[ F(u) = \lim_{\alpha} \, F_{\alpha}(u_{\alpha}). \]
	\end{enumerate}
\end{definition}

\begin{definition}[cf. {\cite[Definition 2.11]{Kuwae-Shioya}}]\label{d:M-c}$($Mosco convergence$)$ \ 
	We say that a net $\{ \mathcal{E}_{\alpha} \}_{\alpha \in \mathcal{A}}$ of closed quadratic forms 
	with $\mathcal{E}_{\alpha}$ on $H_{\alpha}$ Mosco converges	to a closed quadratic form $\mathcal{E}$ on $H$ 
	if and only if both $(\mathrm{F}2)$ and the following $(\mathrm{F}1^{\prime})$ hold:
	\begin{enumerate}
	\item[$(\mathrm{F}1^{\prime})$]
			If a net $\{ u_{\alpha} \}_{\alpha \in \mathcal{A}}$ with $u_{\alpha} \in H_{\alpha}$ weakly converges to $u \in H$, then
				\[ \mathcal{E}(u) \leq \liminf_{\alpha} \, \mathcal{E}_{\alpha}(u_{\alpha}). \]
	\end{enumerate}
\end{definition}

\begin{definition}[cf. {\cite[Definition 2.12]{Kuwae-Shioya}}]$($Asymptotic compactness$)$ \
	The net $\{ \mathcal{E}_{\alpha} \}_{\alpha \in \mathcal{A}}$ is said to be asymptotically compact
	if for any net $\{ u_{\alpha} \}_{\alpha \in \mathcal{A}}$ such that
	$u_{\alpha} \in H_{\alpha}$ and $\limsup_{\alpha} \, ( \mathcal{E}_{\alpha}(u_{\alpha}) + \| u_{\alpha} \|_{H_{\alpha}}^2) < +\infty$,
	there exists a strongly convergent subnet of $\{ u_{\alpha} \}_{\alpha \in \mathcal{A}}$.
\end{definition}

\medskip

\begin{lemma}[cf. {\cite[Lemma 2.15]{Kuwae-Shioya}}]\label{l:g-M}
	Assume that $\{ \mathcal{E}_{\alpha} \}_{\alpha \in \mathcal{A}}$ is asymptotically compact.
	Then, $\{ \mathcal{E}_{\alpha} \}_{\alpha \in \mathcal{A}}$ $\Gamma$-converges to $\mathcal{E}$ 
	if and only if $\{ \mathcal{E}_{\alpha} \}_{\alpha \in \mathcal{A}}$ Mosco converges to $\mathcal{E}$.
\end{lemma}

\medskip

\begin{definition}[cf. {\cite[Definition 2.13]{Kuwae-Shioya}}]$($Compactly convergence$)$ \
	We say that $\mathcal{E}_{\alpha} \to \mathcal{E}$ compactly 
	if $\{ \mathcal{E}_{\alpha} \}_{\alpha \in \mathcal{A}}$ Mosco converges to $\mathcal{E}$
	and if $\{ \mathcal{E}_{\alpha} \}_{\alpha \in \mathcal{A}}$ is asymptotically compact.
\end{definition}

\bigskip

Lastly, we recall the relations between the convergence of densely defined closed quadratic forms $\mathcal{E}_{\alpha}$ 
and a behavior infinitesimal generators $A_{\alpha}$ associated with $\mathcal{E}_{\alpha}$.
For a densely defined closed quadratic form $\mathcal{E}$,
$A$, $\{ T_t \}_{t \geq 0}$ and $\{ R_{\zeta} \}_{\zeta \in \rho(A)}$ denote
the infinitesimal generator associated with $\mathcal{E}$, the instead of the strongly continuous contraction semigroup, and
the strongly continuous resolvent, that is $R_{\zeta} = (A - \zeta)^{-1}$, where $\rho(A)$ denotes the resolvent set of $A$.

\begin{theorem}[cf. {\cite[Theorem 2.4]{Kuwae-Shioya}}]\label{t:K-S}
	The following are all equivalent:
	\begin{enumerate}
	\item $\mathcal{E}_{\alpha} \to \mathcal{E}$ with respect to the Mosco topology 
			(resp. $\mathcal{E}_{\alpha} \to \mathcal{E}$ compactly).
	\item $R_{\zeta}^{\alpha} \to R_{\zeta}$ strongly (resp. compactly) for some $\zeta < 0$.
	\item $T_t^{\alpha} \to T_t$ strongly (resp. compactly) for some $t > 0$.
	\end{enumerate}
\end{theorem}

\section{Gromov-Hausdorff convergence of thin domains}
\subsection{Notations about graphs and thin domains} \ \\[-0.2cm]

We write a point $x \in \mathbb{R}^n$ like $x = {^t\!(}x_1, x^{\prime}) \in \mathbb{R} \times \mathbb{R}^{n-1} = \mathbb{R}^n$ and
denote by $O$ the origin.

Let $G = ( V, E )$ be a graph where $V$ is the set of all vertices and $E$ is the set of all edges as follows:  
\begin{equation*}\label{d:G}
	V = \{ O, v_j \mid j = 1, \ldots, N \}, \qquad E = \{ e_j = \overrightarrow{Ov_j} \mid j = 1, \ldots, N \}.
\end{equation*}
$l_j \in (0, +\infty)$ denotes the length of $e_j$ for $j=1,\ldots,N$.
We identify each edge $e_j$ by an interval $(0, l_j) = \{ s \in \mathbb{R} \mid 0 < s < l_j \}$.
Throughout this paper $s$ is the coordinate of the graph. 
Moreover $d\mu$ denotes the 1 dimensional Lebesgue measure on the graph $G$.

Next, we define a thin domain $\Omega_{\varepsilon}$.
For $j = 1, \ldots, N$, we take an orthogonal matrix $R_j \in O(n)$ satisfying
\begin{equation}\label{d:Rj}
	\det R_j = 1, \qquad R_j \mbox{\boldmath{$a$}} = l_j^{-1} \, \overrightarrow{Ov_j},
\end{equation}
where $\mbox{\boldmath{$a$}} = {^t\!(1,0,\cdots,0)} \in \mathbb{R}^n$ is the unit vector.
Firstly we define tubular domains $D_{j,\varepsilon}$ as
\begin{equation*}\label{d:Dje}
	D_{j,\varepsilon} := \{ x = R_j y \in \mathbb{R}^n \mid \varepsilon l \leq y_1 < l_j, \ |y^{\prime}| < \varepsilon \},
\end{equation*}
for $\varepsilon \in I := (0, \varepsilon_0]$ and $j=1,\ldots,N$.
Here we take positive constants $l$ and $\varepsilon_0$ such that 
$D_{i,\varepsilon} \cap D_{j,\varepsilon} = \emptyset$ for $i \neq j$ and $\varepsilon \in I$.
Now, we shall denote by $B_{\varepsilon}$ the normal section of the tubular domain $D_{j,\varepsilon}$, that is
\begin{equation}\label{d:Br}
	B_r := \{ y^{\prime} \in \mathbb{R}^{n-1} \mid |y^{\prime}| < r \}
\end{equation}
for $r > 0$.

Secondly we define a part of boundary $\partial D_{j,\varepsilon}$ as follows:
\begin{align*}
	\Gamma_{j,\varepsilon} &:= \{ x = R_j y \in \mathbb{R}^n \mid y_1 = l_j, \ |y^{\prime}| \leq \varepsilon \}, \\[0.2cm]
	\Gamma_{j,\varepsilon}^{\prime} &:= \{ x = R_j y \in \mathbb{R}^n \mid y_1 = \varepsilon l, \ |y^{\prime}| \leq \varepsilon \}.
\end{align*}
\begin{figure}[htbp]
	\begin{minipage}{0.4\hsize}
		\begin{center}
		\includegraphics[width=47mm]{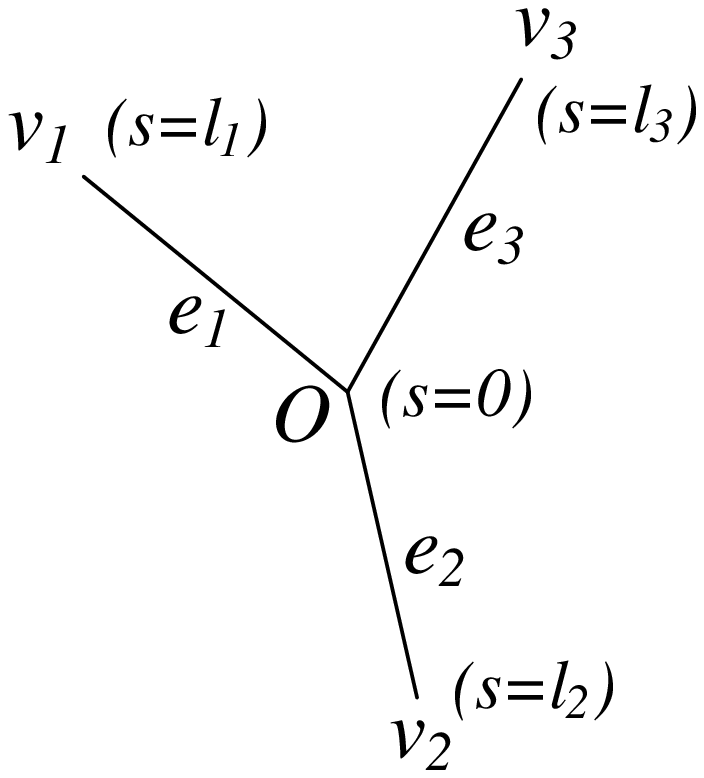}
		\caption{graph $G$}
		\end{center}
	\end{minipage}
	\begin{minipage}{0.58\hsize}
		\begin{center}
		\includegraphics[width=92mm]{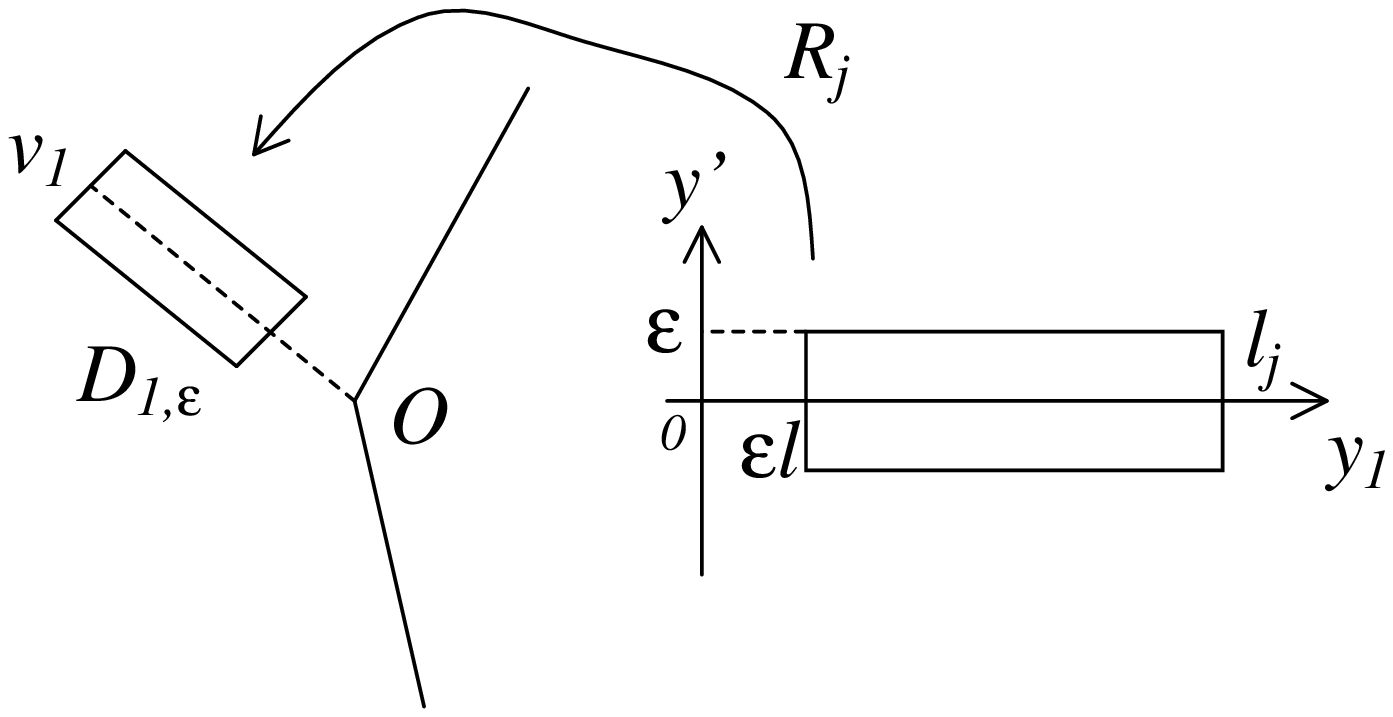}
		\caption{tubular domain $D_{j,\varepsilon}$}
		\end{center}
	\end{minipage}
\end{figure}
We use  the following notations $D_j, \Gamma_j, \ldots$ instead of
$D_{j,\varepsilon_0}, \Gamma_{j,\varepsilon_0}, \ldots$ for simplicity.

Thirdly we fix an open set $J$ in $\mathbb{R}^n$ satisfying the following conditions:
\[ O \in J, \qquad J \cap D_j = \emptyset, \qquad
	\partial J \cap \partial D_j = \Gamma_j^{\prime}
	\qquad \mathrm{for} \ j = 1, \ldots, N \]
and we define a domain $\Omega$ in $\mathbb{R}^n$ and a part of its boundary $\Sigma$ as
\[ \Omega := J \cup \bigcup_{j=1}^{N} D_j, \qquad \Sigma := \partial \Omega \setminus \bigcup_{j=1}^{N} \Gamma_j. \]
Now we suppose that the surface $\Sigma$ is $C^1$.

Lastly we define a family of thin domains for $\varepsilon \in I$ as
\begin{equation}\label{d:Oe}
	J_{\varepsilon} := \{ x = (\varepsilon/\varepsilon_0) z \in \mathbb{R}^n \mid z \in J \}, \qquad
	\Omega_{\varepsilon} := J_{\varepsilon} \cup \bigcup_{j=1}^{N} D_{j, \varepsilon}, \qquad
	\Sigma_{\varepsilon} := \partial \Omega_{\varepsilon} \setminus \bigcup_{j=1}^{N} \Gamma_{j,\varepsilon}.
\end{equation}
We remark that the boundary $\Sigma_{\varepsilon}$ is also $C^1$ and $\displaystyle \bigcap_{\varepsilon \in I} \Omega_{\varepsilon} = G$.
\begin{figure}[htbp]
	\begin{minipage}{0.49\hsize}
		\begin{center}
		\includegraphics[width=75mm]{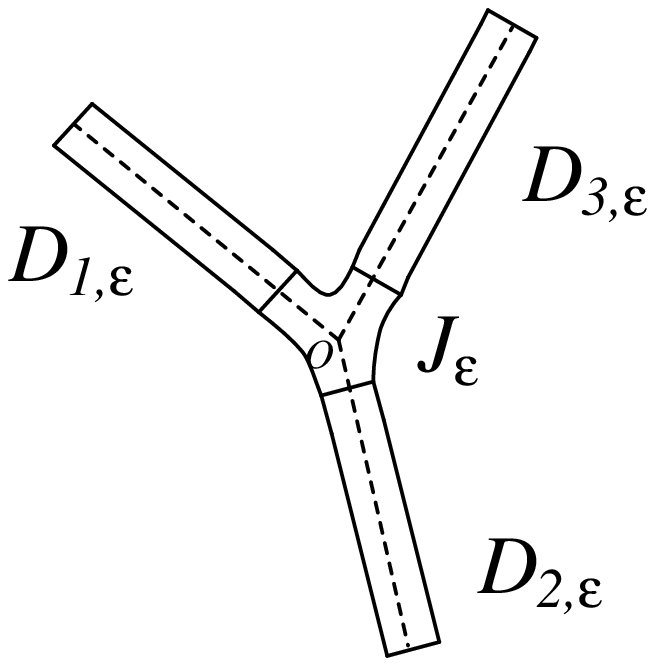}
		\caption{thin domain $\Omega_{\varepsilon}$}
		\end{center}
	\end{minipage}
	\begin{minipage}{0.49\hsize}
		\begin{center}
		\includegraphics[width=75mm]{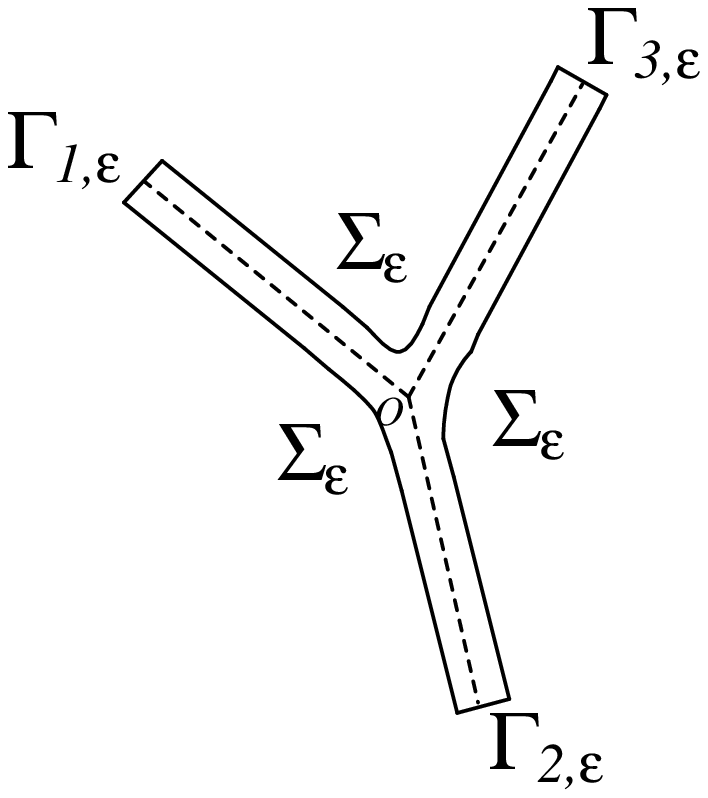}
		\caption{boundary $\partial \Omega_{\varepsilon}$}
		\end{center}
	\end{minipage}
\end{figure}

By the definition of domains we obtain that the volume of each domain is
\begin{equation}\label{e:ordJD}
	|J_{\varepsilon}| = (\varepsilon/\varepsilon_0)^n |J| = O(\varepsilon^n), \qquad
	|D_{j,\varepsilon}| = (l_j - \varepsilon l)\omega \varepsilon^{n-1} = O(\varepsilon^{n-1}),
\end{equation}
where $\omega$ denotes the volume of the $(n-1)$ dimensional unit ball $B_1$ defined by (\ref{d:Br}).

\subsection{Proof of the convergence of thin domains} \ \\[-0.2cm]

We define a Radon measure $d\mu_{\varepsilon}$ on $\Omega_{\varepsilon}$ for $\varepsilon \in I$ as 
\begin{equation}\label{d:dmue}
	d\mu_{\varepsilon} := \frac{1}{\omega \varepsilon^{n-1}} \, dx,
\end{equation}
where $dx$ denotes the Lebesgue measure on $\mathbb{R}^n$.

We define projections $\{ f_{\varepsilon} : \Omega_{\varepsilon} \rightarrow G \}_{\varepsilon \in I}$ as follows.
First, we fix a continuous function $f : \overline{J} \rightarrow G \cap \overline{J}$ such that
\[ \mathrm{Range}(f|_{\Gamma_j^{\prime}}) = G \cap \Gamma_j^{\prime}, \]
where $\overline{J}$ is the closure of $J$.
Next, we define a continuous function $f_{\varepsilon} : \Omega_{\varepsilon} \rightarrow G$ for $\varepsilon \in I$ as
\begin{equation*}
	f_{\varepsilon}(x) := \left\{ \begin{array}{cl}
	\displaystyle \frac{\varepsilon}{\varepsilon_0} \, f \left( \frac{\varepsilon_0}{\varepsilon}x \right) & 
		\mathrm{if} \ \ x \in J_{\varepsilon} = (\varepsilon/\varepsilon_0)J, \\[0.4cm]
		\pi_1(R_j^{-1}x) & \mathrm{if} \ \ x \in D_{j,\varepsilon}
	\end{array} \right.
\end{equation*}
where $\pi_1 : \mathbb{R}^n \rightarrow \mathbb{R}$ is the orthogonal projection with respect to the first component, 
that is $\pi_1(y) = y_1$. 

\begin{proposition}\textbf{(Gromov-Hausdorff convergence of thin domains)}
	The sequence of pointed measured spaces $\{ (\Omega_{\varepsilon}, O, d\mu_{\varepsilon}) \}_{\varepsilon \in I}$ 
	defined by (\ref{d:Oe}) and (\ref{d:dmue}) converges to the pointed measured space $(G, O, d\mu)$ 
	as $\varepsilon \to +0$ in the sense of pointed, measured and compact Gromov-Hausdorff topology.
\end{proposition}

\begin{proof}
	For any test function $\psi \in C_0(G)$, we consider the following limit
	\begin{equation}\label{e:d-GH1}
		\lim_{\varepsilon \to +0} \, \int_{\Omega_{\varepsilon}} \psi \circ f_{\varepsilon} \, d\mu_{\varepsilon}
			= \lim_{\varepsilon \to +0} \left\{ \int_{J_{\varepsilon}} \psi \circ f_{\varepsilon} \, d\mu_{\varepsilon}
				+ \sum_{j=1}^N \frac{1}{\omega {\varepsilon}^{n-1}} \, 
					\int_{D_{j,\varepsilon}} (\psi \circ f_{\varepsilon})(x) \, dx \right\}.
	\end{equation}

	By the definition of $d\mu_{\varepsilon}$, see (\ref{d:dmue}), at the first term of (\ref{e:d-GH1}) we have
	\[ \left| \int_{J_{\varepsilon}} \psi \circ f_{\varepsilon} \, d\mu_{\varepsilon} \right|
		\leq \| \psi \|_{L^{\infty}(G)} \frac{|J_{\varepsilon}|}{\omega {\varepsilon}^{n-1}}, \]
	then 
	\begin{equation}\label{e:d-GH2}
		\lim_{\varepsilon \to +0} \, \int_{J_{\varepsilon}} \psi \circ f_{\varepsilon} \, d\mu_{\varepsilon} = 0
	\end{equation}
	holds because of (\ref{e:ordJD}).

	At the second term of (\ref{e:d-GH1}), $\psi_j = \psi|_{e_j}$ denote the restriction to each edge $e_j$.
	By the transformation of variables $x = R_j y$ we obtain that
	\[ \begin{split}
		\frac{1}{\omega {\varepsilon}^{n-1}} \, \int_{D_{j,\varepsilon}} (\psi \circ f_{\varepsilon})(x) \, dx
			&= \frac{1}{\omega {\varepsilon}^{n-1}} \, \int_{(\varepsilon l,l_j)\times B_{\varepsilon}} \psi_j(y_1) \, dy \\[0.1cm]
			&= \int_{\varepsilon l}^{l_j} \psi_j(y_1) \, dy_1.
	\end{split} \]
	Hence
	\begin{equation}\label{e:d-GH3}
		\lim_{\varepsilon \to +0} \, \frac{1}{\omega {\varepsilon}^{n-1}} \, 
			\int_{D_{j,\varepsilon}} (\psi \circ f_{\varepsilon})(x) \, dx = \int_{0}^{l_j} \psi_j(s) \, ds
	\end{equation}
	holds for $j = 1, \ldots, N$.

	Since (\ref{e:d-GH1}), (\ref{e:d-GH2}) and (\ref{e:d-GH3}), we get the following limit condition:
	\[ \lim_{\varepsilon \to +0} \, \int_{\Omega_{\varepsilon}}  \psi \circ f_{\varepsilon} \, d\mu_{\varepsilon}
		= 0 + \sum_{j=1}^N \int_{0}^{l_j} \psi_j(s) \, ds = \int_{G}  \psi \, d\mu \]
	for any $\psi \in C_0(G)$.
	Therefore the proof is completed.
\end{proof}

\subsection{Definition of energy functionals on thin domains and graphs} \ \\[-0.2cm]

Let $V \in C_0(\mathbb{R}^n)$ be a nonnegative valued function with $\mathrm{supp} \, V \subset \varepsilon_{0}^{-1} J$.
We define a constant $C_V$ and a functional sequence $\{ V_{\varepsilon} \}_{\varepsilon \in I}$ on $\mathbb{R}^n$ as
\begin{equation*}
	C_V := \frac{1}{\omega} \, \int_{\varepsilon_{0}^{-1}J}  V(x) \, dx
\end{equation*}
and
\begin{equation*}
	V_{\varepsilon}(x) := \frac{1}{\varepsilon} V \left( \frac{x}{\varepsilon} \right) \qquad \mathrm{for} \ x \in \mathbb{R}^n.
\end{equation*}
Here the function $V_{\varepsilon}$ has the compact support in $J_{\varepsilon}$.

Next, we define functionals $\varphi_{\varepsilon}^K, \varphi_{\varepsilon}^V, 
\varphi_{\varepsilon} : L^2(\Omega_{\varepsilon}, d\mu_{\varepsilon}) \rightarrow [0,+\infty]$ as
\begin{equation*}\label{d:phie0}
	\varphi_{\varepsilon}^K(u^{\varepsilon}) := \left\{ \begin{array}{cc}
		\displaystyle 
			\int_{\Omega_{\varepsilon}}  |\nabla u^{\varepsilon}|^2 \, d\mu_{\varepsilon} & \quad 
				\mathrm{if} \ \ u^{\varepsilon} \in H^1(\Omega_{\varepsilon},d\mu_{\varepsilon}), \\[0.2cm]
		+\infty & \quad \mathrm{otherwise},
		\end{array} \right.
\end{equation*}
\begin{equation*}
	\varphi_{\varepsilon}^V(u^{\varepsilon}) := \int_{\Omega_{\varepsilon}}  V_{\varepsilon} |u^{\varepsilon}|^2 \, d\mu_{\varepsilon},
\end{equation*}
\begin{equation}\label{d:phie}
	\varphi_{\varepsilon}(u^{\varepsilon}) := \varphi_{\varepsilon}^K(u^{\varepsilon}) + \varphi_{\varepsilon}^V(u^{\varepsilon}) 
\end{equation}
for $\varepsilon \in I$.
Also we define functionals $\varphi^K, \varphi^V, \varphi : L^2(G) \rightarrow [0,+\infty]$ as
\begin{equation*}\label{d:phiN}
	\varphi^K(\psi) := \left\{ \begin{array}{cc}
		\displaystyle \sum_{j=1}^N \int_{0}^{l_j} |\psi_j^{\prime}(s)|^2 \, ds & \quad \mathrm{if} \ \ \psi \in H^1(G), \\[0.2cm]
		+\infty & \quad \mathrm{otherwise},
		\end{array} \right. 
\end{equation*}
\begin{equation*}
	\varphi^V(\psi) := \left\{ \begin{array}{cc}
		C_V |\psi(O)|^2 & \quad \mathrm{if} \ \ \psi \in H^1(G), \\[0.2cm]
		+\infty & \quad \mathrm{otherwise},
		\end{array} \right.
\end{equation*}
\begin{equation}\label{d:phi}
	\varphi(\psi) := \varphi^K(\psi) + \varphi^V(\psi)
\end{equation}
where $\psi_j$ is the restriction to each edge $e_j$.
Here we define functional spaces on $G$ as follows:
\[ \begin{split}
	L^2(G) &:= \{ \psi : G \rightarrow \mathbb{C} \mid \psi_j \in L^2(e_j) \ for \ j=1,\ldots,N \}, \\[0.2cm]
	H^1(G) &:= \{ \psi \in C(G) \mid \psi_j \in H^1(e_j) \ for \ j=1,\ldots,N \},
	\end{split} \]
where $C(G)$ is the set of all continuous functions on $G$.

In this paper we prove that $\varphi_{\varepsilon}$ Mosco converges to $\varphi$ as $\varepsilon \to +0$.
This statement is our main theorem.
We discuss more details in section 5.

\section{The compactness condition}

To prove the $\Gamma$-convergence of our energy functionals, 
we have to study the behavior as $\varepsilon \to 0$ of any functional sequence $\{ u^{\varepsilon} \}_{\varepsilon \in I}$
with $u^{\varepsilon} \in L^2(\Omega_{\varepsilon}, d\mu_{\varepsilon})$ 
satisfying the following condition: there exists a constant $M > 0$ such that
	\begin{equation}\label{e:com}
		\sup_{\varepsilon \in I} \, \left\{ \varphi_{\varepsilon}^K(u^{\varepsilon}) 
			+ \| u^{\varepsilon} \|_{L^2(\Omega_{\varepsilon},d\mu_{\varepsilon})}^2 \right\} \leq M.
	\end{equation}
To do so, we divide the thin domains and consider the problem on each junction domain $J_{\varepsilon}$ and 
tubular domain $D_{j,\varepsilon}$ separately.

\subsection{In the tubular domains} \ \\[-0.2cm]

First, we consider the behavior of $\{ u^{\varepsilon} \}_{\varepsilon \in I}$ in the tubular domain $D_{j,\varepsilon}$.
It is convenient to study the problem in some domain independent to $\varepsilon$.
So we define a fixed tube $Q_j$ in $\mathbb{R}^n$ as
\begin{equation*}\label{d:Qj}
	Q_j := (0,l_j) \times B_1
\end{equation*}
and define a functional sequence $\{ w_j^{\varepsilon} \}_{\varepsilon \in I} \subset L^2(Q_j)$ as follows:
\begin{equation}\label{d:we}
	w_j^{\varepsilon}(y) := u^{\varepsilon} ( R_j \alpha_{\varepsilon}(y) ) \qquad \mathrm{for} \ y \in Q_j,
\end{equation}
where $R_j$ is the orthogonal matrix satisfying (\ref{d:Rj}) and 
$\alpha_{\varepsilon} : Q_j \rightarrow (\varepsilon l, l_j) \times B_{\varepsilon}$ is a transformation of variables defined by
\begin{equation}\label{d:Ae}
	\alpha_{\varepsilon}(y) := {^t}\! \left( \frac{l_j - \varepsilon l}{l_j} \, y_1 + \varepsilon l, \varepsilon y^{\prime} \right)
	\qquad \mathrm{for} \ y = {^t\!(}y_1, y^{\prime}) \in Q_j.
\end{equation}
It is easy to calculate the Jacobian of the transformation of variables $\alpha_{\varepsilon}$:
\[ \det (\nabla \alpha_{\varepsilon}(y)) = (l_j - \varepsilon l)l_j^{-1}\varepsilon^{n-1} \]

\begin{lemma}\label{l:we}
	Suppose that a functional sequence $\{ u^{\varepsilon} \}_{\varepsilon \in I}$ 
	with $u^{\varepsilon} \in L^2(\Omega_{\varepsilon}, d\mu_{\varepsilon})$ satisfies the condition (\ref{e:com}).
	Let $\{ w_j^{\varepsilon} \}_{\varepsilon \in I}$ be the functional sequence defined by (\ref{d:we}). 
	Then the following properties hold.
	\begin{enumerate}
	\item For $j=1,\ldots,N$, it follows that
			\[ \lim_{\varepsilon \to +0} \, \int_{Q_j}  |\nabla_{y^{\prime}} w_j^{\varepsilon}(y)|^2 \, dy = 0, \]
			where $\nabla_{y^{\prime}}$ means the derivative with respect to $y^{\prime}$,
				that is $\displaystyle | \nabla_{y^{\prime}} w_j^{\varepsilon}(y)|^2
					= \sum_{i=2}^{n} \left| \frac{\partial w_j^{\varepsilon}}{\partial y_i} (y) \right|^2$.
	\item There exists a subsequence $\{ w_j^{\varepsilon_m} \}_{m=1}^{\infty}$ of $\{ w_j^{\varepsilon} \}_{\varepsilon \in I}$ such that
			$\{ w_j^{\varepsilon_m} \}_{m=1}^{\infty}$ converges to some function
			$\psi_j^{\infty} = \psi_j^{\infty}(s) \in H^1(0,l_j)$ in $L^2(Q_j)$ for $j=1,\ldots,N$ , that is
				\[ \lim_{m \to \infty} \, \int_{Q_j}  |w_j^{\varepsilon_m}(y) - \psi_j^{\infty}(y_1)|^2 \, dy = 0. \]
	\end{enumerate}
\end{lemma}

\begin{proof}
	By the transformation of variables $y = \alpha_{\varepsilon}^{-1}(R_j^{-1}x)$ defined by (\ref{d:Ae}), we obtain
	\begin{equation*}\label{e:we-1}
		\begin{split}
		\int_{Q_j}  |\nabla_{y^{\prime}} w_j^{\varepsilon}(y)|^2 dy
		&= \frac{l_j}{(l_j - \varepsilon l)\varepsilon^{n-3}} \, 
			\int_{D_{j,\varepsilon}}  |\nabla_{x^{\prime}} u^{\varepsilon}(x)|^2 \, dx \\[0.1cm]
		&= \frac{l_j \omega \varepsilon^2}{l_j - \varepsilon l} \, 
			\int_{D_{j,\varepsilon}}  |\nabla_{x^{\prime}} u^{\varepsilon}|^2 \, d\mu_{\varepsilon} \\[0.1cm]
		&\leq \frac{l_j \omega M \varepsilon^2}{l_j - \varepsilon l}
 		\end{split}
	\end{equation*}
	since the condition (\ref{e:com}). 
	So the claim \ref{l:we} (1) holds.

	Also by the similar calculations we have
	\begin{equation}\label{e:we-2}
		\begin{split}
		\int_{Q_j}  \left| \frac{\partial w_j^{\varepsilon}}{\partial y_1}(y) \right|^2 dy
		&= \frac{l_j - \varepsilon l}{l_j \varepsilon^{n-1}} \, 
			\int_{D_{j,\varepsilon}}  \left| \frac{\partial u^{\varepsilon}}{\partial x_1}(x) \right|^2 dx \\[0.1cm]
		&= \frac{(l_j -\varepsilon l)\omega}{l_j} \, 
			\int_{D_{j,\varepsilon}}  \left| \frac{\partial u^{\varepsilon}}{\partial x_1} \right|^2 d\mu_{\varepsilon} \\[0.1cm]
		&\leq \omega M,
		\end{split}
	\end{equation}
	and
	\begin{equation*}\label{e:we-3}
		\begin{split}
		\int_{Q_j}  |w_j^{\varepsilon}(y)|^2 \, dy
		&= \frac{l_j}{(l_j - \varepsilon l) \varepsilon^{n-1}} \, \int_{D_{j,\varepsilon}}  |u^{\varepsilon}(x)|^2 \, dx \\[0.1cm]
		&= \frac{l_j \omega}{l_j - \varepsilon l} \, \int_{D_{j,\varepsilon}}  |u^{\varepsilon}|^2 \, d\mu_{\varepsilon} \\[0.1cm]
		&\leq \frac{l_j \omega M}{l_j - \varepsilon_0 l}.
		\end{split}
	\end{equation*}

	From these calculations, $\{ w_j^{\varepsilon} \}_{\varepsilon \in I}$ is bounded in $H^1(Q_j)$.
	Therefore there exist a subsequence $\{ \varepsilon_m \}_{m=1}^{\infty} \subset I, \, \varepsilon_m \searrow 0$ 
	and a function $w_j^{\infty} \in H^1(Q_j)$ such that
	\begin{equation*}\label{e:we-4}
		w_j^{\varepsilon_m} \ \rightarrow \ w_j^{\infty} \quad \mathrm{strongly \ in} \ L^2(Q_j), 
	\end{equation*}
	\begin{equation*}\label{e:we-5}
		\nabla w_j^{\varepsilon_m} \ \rightharpoonup \ \nabla w_j^{\infty} \quad \mathrm{weakly \ in} \ L^2(Q_j),
	\end{equation*}
	as $m \to \infty$ for $j=1,\ldots,N$.
	Because of Lemma \ref{l:we}(1), it follows that 
	\[ \nabla_{y^{\prime}} w_j^{\varepsilon_m} \ \rightarrow \ 0 \qquad \mathrm{in} \ L^2(Q_j) \qquad (m \to \infty), \]
	then $\nabla_{y^{\prime}} w_j^{\infty} = 0$,
	which in terms implies that we can take some function $\psi_j^{\infty} = \psi_j^{\infty}(s) \in H^1(0,l_j)$ satisfying
	\begin{equation}\label{d:psi}
		w_j^{\infty}(y) = \psi_j^{\infty}(y_1) \qquad \mathrm{for} \ \ y \in Q_j.
	\end{equation}
	This function $\psi_j^{\infty}$ satisfies the claim in Lemma \ref{l:we}(2).
\end{proof}

\bigskip

Next, we prove that the function $\psi_j^{\infty}$ is the limit of $u^{\varepsilon_m}$ 
in the tubular domain $D_{j,\varepsilon_m}$ as $m \to \infty$.
\begin{lemma}\label{l:u-psi}
	For each $j=1,\ldots,N$, the functional sequence $\{ u^{\varepsilon_m} \}_{m=1}^{\infty}$ converges to
	the function $\psi_j^{\infty}$ defined by (\ref{d:psi})
	in the tubular domain $D_{j,\varepsilon_m}$ as $m \to \infty$, that is
	\begin{equation*}\label{e:u-p1}
		\lim_{m \to \infty} \, 
			\int_{D_{j,\varepsilon_m}}  |u^{\varepsilon_m} - \psi_j^{\infty} \circ f_{\varepsilon_m}|^2 \, d\mu_{\varepsilon_m} = 0.
	\end{equation*}
\end{lemma}

\begin{proof}
	By using the transformation of variables $x = R_j \alpha_{\varepsilon}(y)$, it follows that
	\begin{equation}\label{e:u-p2}
		\begin{split}
		&\int_{D_{j,\varepsilon_m}}  |u^{\varepsilon_m} - \psi_j^{\infty} \circ f_{\varepsilon_m}|^2 \, d\mu_{\varepsilon_m} \\[0.1cm]
		&\qquad\qquad = \frac{a_{j,m}}{\omega}
			\int_{Q_j}  |w_j^{\varepsilon_m}(y) - \psi_j^{\infty}(a_{j,m} y_1 + \varepsilon_m l)|^2 \, dy \\[0.1cm]
		&\qquad\qquad \leq \frac{2a_{j,m}}{\omega} \int_{Q_j} |w_j^{\varepsilon_m}(y) - \psi_j^{\infty}(y_1)|^2 \, dy 
			+ 2a_{j,m} \int_{0}^{l_j} |\psi_j^{\infty}(s) - \psi_j^{\infty}(a_{j,m}s + \varepsilon_m l)|^2 \, ds
		\end{split}
	\end{equation}
	where a constant $a_{j,m}$ is defined by
	\begin{equation}\label{d:ajm}
		a_{j,m} := (l_j - \varepsilon_m l)/l_j
	\end{equation}
	and converges to $1$ as $m \to \infty$.

	Since Lemma \ref{l:we} (2) and $\psi_j^{\infty} \in H^1(0,l_j) \subset C([0,l_j])$, 
	the right hand side of (\ref{e:u-p2}) converges to $0$ as $m \to \infty$.
	Therefore we get the conclusion.
\end{proof}

\bigskip

Hereafter we take the subsequence $\{ u^{\varepsilon_m} \}_{m=1}^{\infty}$, which satisfies the claim in Lemma \ref{l:we} (2),
of $\{ u^{\varepsilon} \}_{\varepsilon \in I}$.
So we obtain the limit functions $\{ \psi_j^{\infty} \}_{j=1}^N$ on edges.
Next we consider the connecting conditions at the origin $O$ about these limit functions.

\subsection{In the junction domain} \ \\[-0.2cm]

Next, we study the behavior of $\{ u^{\varepsilon} \}_{\varepsilon \in I}$, which satisfies the condition (\ref{e:com}), 
in the junction domain $J_{\varepsilon}$.
Especially we consider the continuity of the limit function at the junction point $O$. 
To do so, we fix a constant $a \in (l, \min l_j/\varepsilon_0)$ and define
\begin{equation*}\label{d:Djea}
	D_{j,\varepsilon}^a := \{ x = R_j y \in \mathbb{R}^n \mid \varepsilon l \leq y_1 < \varepsilon a, \ |y^{\prime}| < \varepsilon \},
	\qquad J_{\varepsilon}^a := J_{\varepsilon} \cup \bigcup_{j=1}^{N} D_{j,\varepsilon}^a
\end{equation*}
for $\varepsilon \in I$ and $j=1,\ldots,N$. 
Next, we define a fixed domain $J^a$ as $J^a = \varepsilon^{-1} J_{\varepsilon}^a$ 
and a functional sequence $\{ v^{\varepsilon} \}_{\varepsilon \in I} \subset L^2(J^a)$ as follows:
\begin{equation}\label{d:ve}
	v^{\varepsilon}(z) := u^{\varepsilon} ( \varepsilon z ) \qquad	\mathrm{for} \ z \in J^a = \varepsilon^{-1} J_{\varepsilon}^a.
\end{equation}

Since the junction domain $J_{\varepsilon}$ squeezes to the origin $O$,
we expect that the functional sequence $\{ v^{\varepsilon} \}_{\varepsilon \in I}$ converges to some constant as $\varepsilon \to +0$.

\begin{lemma}\label{l:ve}
	Suppose that a functional sequence $\{ u^{\varepsilon} \}_{\varepsilon \in I}$ 
	with $u^{\varepsilon} \in L^2(\Omega_{\varepsilon}, d\mu_{\varepsilon})$ satisfies the condition (\ref{e:com}).
	Let $\{ v^{\varepsilon} \}_{\varepsilon \in I}$ be the functional sequence defined by (\ref{d:ve}). 
	Then the following properties hold.
	\begin{enumerate}
	\item It follows that
			\[ \lim_{\varepsilon \to +0} \, \int_{J^a}  |\nabla v^{\varepsilon}(z)|^2 \, dz = 0. \]
	\item There exists the sequence $\{ \xi_{\varepsilon} \}_{\varepsilon \in I}$ such that
			\[ \lim_{\varepsilon \to +0} \, \int_{J^a}  |v^{\varepsilon}(z) - \xi_{\varepsilon}|^2 \, dz = 0, \]
			where $\xi_{\varepsilon}$ is defined by
			\begin{equation}\label{d:xie}
				\xi_{\varepsilon} := \frac{1}{|J^a|} \, \int_{J^a}  v^{\varepsilon}(z) \, dz.
			\end{equation}
	\end{enumerate}
\end{lemma}

\begin{proof}
	By using the transformation of variables $z = \varepsilon^{-1}x$, we have
	\[ \begin{split}
		\int_{J^a}  |\nabla v^{\varepsilon}(z)|^2 \, dz
		&= \frac{1}{\varepsilon^{n-2}} \, \int_{J_{\varepsilon}^a}  |\nabla u^{\varepsilon}(x)|^2 \, dx \\[0.1cm]
		&= \omega \varepsilon \int_{J_{\varepsilon}^a}  |\nabla u^{\varepsilon}|^2 \, d\mu_{\varepsilon} \\[0.1cm]
		&\leq M \omega \varepsilon.
		\end{split} \]
	So Lemma \ref{l:ve} (1) holds.

	The complex number $\xi_{\varepsilon}$ denotes the mean value of $v^{\varepsilon}$ in $J^a$ for $\varepsilon \in I$.
	Then we obtain the claim in Lemma \ref{l:ve} (2) by using the Poincar\'{e} inequality and Lemma \ref{l:ve} (1).
\end{proof}

\subsection{Continuity of the limit function} \ \\[-0.2cm]

In this part we mention that the limit function $\{ \psi_j^{\infty} \}_{j=1}^N$ belongs to the effective domain of $\varphi$
defined by (\ref{d:phi}).

\begin{lemma}\label{l:psi-c}
	Suppose that a functional sequence $\{ u^{\varepsilon} \}_{\varepsilon \in I}$ 
	with $u^{\varepsilon} \in L^2(\Omega_{\varepsilon}, d\mu_{\varepsilon})$ satisfies
	the condition (\ref{e:com}).
	Then there exists a subsequence $\{ u^{\varepsilon_k} \}_{k=1}^{\infty}$ of $\{ u^{\varepsilon} \}_{\varepsilon \in I}$
	and a function $\psi^{\infty} \in H^1(G)$ satisfying the following condition
	\begin{equation}\label{e:u-P}
		\lim_{k \to \infty} \, 
			\int_{\Omega_{\varepsilon_k}} |u^{\varepsilon_k} - \psi^{\infty} \circ f_{\varepsilon_k}|^2 \, d\mu_{\varepsilon_k} = 0.
	\end{equation}
\end{lemma}

\begin{proof}
	First, we take a sequence $\{ \varepsilon_m \}_{m=1}^{\infty} \subset I$ ($\varepsilon_m \searrow 0$) such that  
	all claims in Lemma \ref{l:we} and Lemma \ref{l:ve} are satisfied.

	To show that $\psi_j^{\infty}(+0)$ is independent to $j$, we consider the following limit:
	\begin{equation*}\label{e:psi-c1}
		\lim_{m \to \infty} \, \frac{\omega}{\varepsilon_m} \, 
			\int_{\varepsilon_m l}^{\varepsilon_m a} |\psi_j^{\infty}(s) - \xi_{\varepsilon_m}|^2 \, ds,
	\end{equation*}
	where $\{ \xi_{\varepsilon} \}$ is the sequence defined by (\ref{d:xie}).
	Here, we have
	\begin{equation}\label{e:psi-c3}
		\begin{split}
		&\frac{\omega}{\varepsilon_m} \, 
			\int_{\varepsilon_m l}^{\varepsilon_m a} |\psi_j^{\infty} (s) - \xi_{\varepsilon_m}|^2 \, ds \\[0.1cm]
			&\qquad\qquad\qquad = \frac{1}{\varepsilon_m^n} \, \int_{\varepsilon_m l}^{\varepsilon_m a} \int_{B_{\varepsilon_m}} 
				|\psi_j^{\infty} (y_1) - \xi_{\varepsilon_m}|^2 \, dy \\[0.1cm]
			&\qquad\qquad\qquad \leq \frac{2}{\varepsilon_m^n} \, \int_{D_{j,\varepsilon_m}^a}  
				( |\psi_j^{\infty} (\pi_1 R_j^{-1}x) - u^{\varepsilon_m}(x)|^2 + |u^{\varepsilon_m}(x) - \xi_{\varepsilon_m}|^2 ) \, dx.
		\end{split}
	\end{equation}
	We know already that the second term of the right hand side of (\ref{e:psi-c3}) converges to $0$ as $m \to \infty$ 
	because of Lemma \ref{l:ve}.
	In fact
	\begin{equation}\label{e:psi-c13}
		\frac{1}{\varepsilon_m^n} \, \int_{D_{j,\varepsilon_m}^a}  |u^{\varepsilon_m}(x) - \xi_{\varepsilon_m}|^2 \, dx
		\leq \int_{J^a}  |v^{\varepsilon_m}(z) - \xi_{\varepsilon_m}|^2 \, dz \quad \longrightarrow \quad 0 \qquad (m \to \infty).
	\end{equation}
	At the first term of the right hand side of (\ref{e:psi-c3}), 
	we use the transformation of variables $x = R_j \alpha_{\varepsilon_m}(y)$.
	Then
	\begin{equation}\label{e:psi-c4}
		\begin{split}
		&\frac{1}{\varepsilon_m^n} \, 
			\int_{D_{j,\varepsilon_m}^a}  |\psi_j^{\infty} (\pi_1 R_j^{-1}x) - u^{\varepsilon_m}(x)|^2 \, dx \\[0.1cm]
		&\qquad\qquad\qquad\qquad\qquad = \frac{a_{j,m}}{\varepsilon_m} \, \int_{0}^{a_{j,m}^{-1}(a-l) \varepsilon_m} \int_{B_1} 
			| \psi_j^{\infty} ( a_{j,m} y_1 + \varepsilon_m l ) - w_j^{\varepsilon_m}(y) |^2 \, dy^{\prime}dy_1
		\end{split} 
	\end{equation}
	where $a_{j,m}$ is the constant defined by (\ref{d:ajm}).
	By taking a new coordinate $s = \frac{a_{j,m} y_1}{(a-l)\varepsilon_m}$, 
	the right hand side of (\ref{e:psi-c4}) equals to
	\begin{equation}\label{e:psi-c5}
		(a-l) \int_{0}^{1} \int_{B_1}  | \psi_j^{\infty}(g_m(s)) - w_j^{\varepsilon_m}(h_m(s), y^{\prime}) |^2 \, dy^{\prime}ds
	\end{equation}
	where
	\[ g_m(s) := \varepsilon_m \{ (a-l)s + l \}, \qquad h_m(s) := a_{j,m}^{-1} \varepsilon_m (a-l) s. \]

	Firstly, it follows that
	\begin{equation}\label{e:psi-c6}
		\lim_{m \to \infty} \, \int_{0}^{1} |\psi_j^{\infty}(g_m(s)) - \psi_j^{\infty}(0) |^2 \, ds = 0
	\end{equation}
	since $\psi_j^{\infty} \in C([0,l_j])$ and $g_m(s) \to 0$ as $m \to \infty$.

	Secondly, $\{ w_j^{\varepsilon_m} \}_{m=1}^{\infty} \subset H^1(Q_j)$ is bounded,
	then $\{ w_j^{\varepsilon_m}|_{\partial Q_j} \}_{m=1}^{\infty} \subset H^{\frac{1}{2}}(\partial Q_j)$ is relatively compact 
	in $L^2(\partial Q_j)$. 
	So $w_j^{\varepsilon_m}|_{\partial Q_j}$ converges to $w_j^{\infty}|_{\partial Q_j}$ in $L^2(\partial Q_j)$.
	Therefore we get
	\begin{equation}\label{e:psi-c7}
		\lim_{m \to \infty} \, \int_{B_1}  |\psi_j^{\infty}(0) - w_j^{\varepsilon_m}(0,y^{\prime})|^2 \, dy^{\prime} = 0.
	\end{equation}

	Thirdly, we have
	\begin{equation*}
		\begin{split}
		\int_{0}^{1} \int_{B_1}  | w_j^{\varepsilon_m}(0, y^{\prime}) - w_j^{\varepsilon_m}(h_m(s), y^{\prime}) |^2 \, dy^{\prime}ds
		&\leq \int_{0}^{1} \int_{B_1}  h_m(s) 
			\int_{0}^{h_m(s)} \left| \frac{\partial w_j^{\varepsilon_m}}{\partial y_1}(\tau, y^{\prime}) \right|^2
				d\tau dy^{\prime}ds \\
		&\leq a_{j,m} \varepsilon_m (a-l) \int_{0}^{1} s \int_{B_1}  \int_{0}^{l_j}
				\left| \frac{\partial w_j^{\varepsilon_m}}{\partial y_1}(\tau, y^{\prime}) \right|^2 d\tau dy^{\prime}ds \\
		&\leq a_{j,m} \varepsilon_m (a-l) \int_{Q_j}  \left| \frac{\partial w_j^{\varepsilon_m}}{\partial y_1}(y) \right|^2 dy \\
		&\leq a_{j,m} \varepsilon_m (a-l) \omega M
		\end{split}
	\end{equation*}
	since (\ref{e:we-2}).
	Therefore we obtain
	\begin{equation}\label{e:psi-c8}
		\lim_{m \to \infty} \, 
		\int_{0}^{1} \int_{B_1}  | w_j^{\varepsilon_m}(0, y^{\prime}) - w_j^{\varepsilon_m}(h_m(s), y^{\prime}) |^2 \, dy^{\prime}ds = 0.
	\end{equation}

	From (\ref{e:psi-c5}), (\ref{e:psi-c6}), (\ref{e:psi-c7}) and (\ref{e:psi-c8}), it follows that
	\begin{equation}\label{e:psi-c9}
		\lim_{m \to \infty} \, \frac{1}{\varepsilon_m^n} \, \int_{D_{j,\varepsilon_m}^a} 
		|\psi_j^{\infty} (\pi_1 R_j^{-1}x) - u^{\varepsilon_m}(x)|^2 \, dx = 0.
	\end{equation}
	So the following limit
	\begin{equation}\label{e:psi-c10}
		\lim_{m \to \infty} \, \frac{1}{\varepsilon_m} \, 
			\int_{\varepsilon_m l}^{\varepsilon_m a} |\psi_j^{\infty} (s) - \xi_{\varepsilon_m}|^2 \, ds = 0
	\end{equation}
	holds by (\ref{e:psi-c3}), (\ref{e:psi-c13}) and (\ref{e:psi-c9}) for $j=1,\ldots,N$.

\medskip

	Moreover it follows that
	\begin{equation}\label{e:psi-c14}
		\frac{1}{\varepsilon_m} \, \int_{\varepsilon_m l}^{\varepsilon_m a} |\psi_i^{\infty}(s) - \psi_j^{\infty}(s)|^2 \, ds
		\leq \frac{2}{\varepsilon_m} \, \int_{\varepsilon_m l}^{\varepsilon_m a} ( |\psi_i^{\infty}(s) - \xi_{\varepsilon_m}|^2
				+ |\xi_{\varepsilon_m} - \psi_j^{\infty}(s)|^2 ) \, ds
	\end{equation}
	and
	\begin{equation}\label{e:psi-c15}
		\begin{split}
			\lim_{m \to \infty} \, \frac{1}{\varepsilon_m} \, 
				\int_{\varepsilon_m l}^{\varepsilon_m a} |\psi_i^{\infty}(s) - \psi_j^{\infty}(s)|^2 \, ds
			&= \lim_{m \to \infty} \, 
		\int_{l}^{a} |\psi_i^{\infty}(\varepsilon_m s^{\prime}) - \psi_j^{\infty}(\varepsilon_m s^{\prime})|^2 \, ds^{\prime} \\[0.1cm]
			&= (a-l)|\psi_i^{\infty}(+0) - \psi_j^{\infty}(+0)|^2.
		\end{split}
	\end{equation}
	Hence we obtain
	\[ |\psi_i^{\infty}(+0) - \psi_j^{\infty}(+0)| = 0 \]
	for $i,j=1,\ldots,N$ since (\ref{e:psi-c10}), (\ref{e:psi-c14}) and (\ref{e:psi-c15}).
	Then it follows that
	\begin{equation}\label{e:psi-c11}
		\psi_1^{\infty}(+0) = \cdots = \psi_N^{\infty}(+0) = \lim_{k \to \infty} \, \xi_{\varepsilon_{m_k}} =: v^{\infty}
	\end{equation}
	by taking some suitable subsequence $\{ \xi_{\varepsilon_{m_k}} \}_{k=1}^{\infty}$ of $\{ \xi_{\varepsilon_m} \}_{m=1}^{\infty}$.
	Therefore we can define a continuous function $\psi^{\infty} \in H^1(G)$ as
	\begin{equation*}
		\psi^{\infty} := \left\{ \begin{array}{cl}
			\psi_j^{\infty}(s) & \qquad \mathrm{on} \ \ e_j = \{ s \mid 0 < s < l_j \} \quad \mathrm{for} \ j=1,\ldots,N, \\[0.2cm]
			v^{\infty} & \qquad \mathrm{at} \ \ O.
			\end{array} \right.
	\end{equation*}	

	Lastly, we prove (\ref{e:u-P}).
	Here we replace $\varepsilon_{m_k}$ to $\varepsilon_k$ for simplicity.
	Because of Lemma \ref{l:u-psi}, we have to prove only this limit condition:
	\begin{equation}\label{e:u-P1}
		\lim_{k \to \infty} \, 
			\int_{J_{\varepsilon_k}} |u^{\varepsilon_k} - \psi^{\infty} \circ f_{\varepsilon_k}|^2 \, d\mu_{\varepsilon_k} = 0.
	\end{equation}
	Now, we have
	\[ \begin{split}
		\int_{J_{\varepsilon_k}}  |u^{\varepsilon_k} - \psi^{\infty} \circ f_{\varepsilon_k}|^2 \, d\mu_{\varepsilon_k}
		&\leq 2\int_{J_{\varepsilon_k}} ( |u^{\varepsilon_k} - \xi_{\varepsilon_k}|^2 
			+ |\xi_{\varepsilon_k} - \psi^{\infty} \circ f_{\varepsilon_k}|^2 ) \, d\mu_{\varepsilon_k} \\[0.1cm]
		&\leq \frac{2 \varepsilon_k}{\omega} \, \int_{J^a}  |v^{\varepsilon_k}(z) - \xi_{\varepsilon_k}|^2 \, dz
			+ 2 \| \xi_{\varepsilon_k} - \psi^{\infty} \|_{L^{\infty}(G)}^2 \frac{|J_{\varepsilon_k}|}{\omega \varepsilon_k^{n-1}}.
		\end{split} \]
	Hence (\ref{e:u-P1}) holds since (\ref{e:ordJD}), (\ref{e:psi-c11}) and Lemma \ref{l:ve} (2).
\end{proof}

\section{Convergences of energy functionals}

We discuss our main theorem \ref{p:M-c} in this section.

\begin{theorem}\label{p:G-c}
	The functional sequence $\{ \varphi_{\varepsilon} \}_{\varepsilon \in I}$ defined by (\ref{d:phie}) $\Gamma$-converges to
	the function $\varphi$ defined by (\ref{d:phi}) as $\varepsilon \to +0$ in the sense of Definition \ref{d:g-c}.
\end{theorem}

\begin{proof}
	We have to prove that the two conditions $(\mathrm{F}1)$ and $(\mathrm{F}2)$ in Definition \ref{d:g-c} are satisfied.

	First, we check $(\mathrm{F}1)$.
	To do so, let $\{ u^{\varepsilon} \}_{\varepsilon \in I}$ with 
	$u^{\varepsilon} \in L^2(\Omega_{\varepsilon}, d\mu_{\varepsilon})$ be any functional sequence which converges strongly
	to some function $\psi \in L^2(G)$ as $\varepsilon \to +0$ in the sense of Definition \ref{d:st-c}.
	Now we have to show that
	\begin{equation}\label{e:g-11}
		\varphi(\psi) \leq \liminf_{\varepsilon \to +0} \, \varphi_{\varepsilon}(u^{\varepsilon}).
	\end{equation}
	If $\displaystyle \liminf_{\varepsilon \to +0} \, \varphi_{\varepsilon}(u^{\varepsilon}) = +\infty$, then (\ref{e:g-11}) is trivial.
	So we can assume that 
	\begin{equation*}\label{e:g-12}
		\liminf_{\varepsilon \to +0} \, \varphi_{\varepsilon}(u^{\varepsilon}) < +\infty.
	\end{equation*}
	Hence we are able to take some suitable subsequence $\{ u^{\varepsilon_m} \}_{m=1}^{\infty}$ for $\varepsilon_m \searrow 0$ which
	there exists a constant $C > 0$ such that
	\[ \varphi_{\varepsilon_m}(u^{\varepsilon_m}) \leq C \quad (m \in \mathbb{N}), 
		\qquad \liminf_{\varepsilon \to +0} \, \varphi_{\varepsilon}(u^{\varepsilon})
		= \liminf_{m \to \infty} \, \varphi_{\varepsilon_m}(u^{\varepsilon_m}). \]
	In this case, $\{ u^{\varepsilon_m} \}_{m=1}^{\infty}$ satisfies the condition (\ref{e:com}), that is 
	\[ \sup_{m \in \mathbb{N}} \left\{ \varphi_{\varepsilon_m}^K(u^{\varepsilon_m}) 
		+ \| u^{\varepsilon_m} \|_{L^2(\Omega_{\varepsilon_m},d\mu_{\varepsilon_m})}^2 \right\} < +\infty. \]
	By Lemma \ref{l:psi-c}, some subsequence of $\{ u^{\varepsilon_m} \}_{m=1}^{\infty}$ converges strongly to a function
	$\psi^{\infty} \in H^1(G)$ in the sense of Definition \ref{d:st-c}.
	Therefore it follows that
	\begin{equation*}
		\psi = \psi^{\infty} \in H^1(G), \qquad
			\lim_{m \to \infty} \, 
				\| u^{\varepsilon_m} - \psi \circ f_{\varepsilon_m} \|_{L^2(\Omega_{\varepsilon_m},d\mu_{\varepsilon_m})} = 0.
	\end{equation*}	
		
	We fix a positive constant $\delta$.
	Then, for $0 < \varepsilon < \delta/l$ we can divide domains $\Omega_{\varepsilon}$ as follows:
	\[ \Omega_{\varepsilon} = J_{\varepsilon}^{\delta} \sqcup \bigsqcup_{j=1}^{N} D_{j,\varepsilon}^{\delta}, \qquad
		D_{j,\varepsilon}^{\delta} := \{ x = R_j y \in \mathbb{R}^n \mid \delta < y_1 < l_j, \, |y^{\prime}| < \varepsilon \}. \]
	Here we prepare a useful lemma about functionals on tubular domains.
	This lemma is proved by only direct calculations. 
	So we omit it here and prove in the end of this proposition's proof.

	\begin{lemma}\label{l:tube-lsc}
		Let $Q_{\varepsilon} := (0,L) \times B_{\varepsilon}$ be a tubular domain and $E_{\varepsilon}$ be a following bilinear form 
		on $L^2(Q_{\varepsilon},d\mu_{\varepsilon})$:
		\[ E_{\varepsilon}(g^{\varepsilon}) := \left\{ \begin{array}{cc}
			\displaystyle 
				\int_{Q_{\varepsilon}} |\nabla g^{\varepsilon}|^2 \, d\mu_{\varepsilon} & \quad 
					\mathrm{if} \ \ g^{\varepsilon} \in H^1(Q_{\varepsilon},d\mu_{\varepsilon}), \\[0.2cm]
				+\infty & \quad \mathrm{otherwise}
		\end{array} \right. \]
		for $0 < \varepsilon < 1$.
		Suppose that a functional sequence $\{ g^{\varepsilon} \}_{0 < \varepsilon < 1}$ 
		with $g^{\varepsilon} \in L^2(Q_{\varepsilon}, d\mu_{\varepsilon})$ converges strongly
			to some function $g \in L^2(0,L)$ as $\varepsilon \to +0$ in the sense of Definition \ref{d:st-c}
		and $\{ E_{\varepsilon}(g^{\varepsilon}) \}_{0 < \varepsilon < 1}$ is bounded.
		Then $g \in H^1(0,L)$ and the following inequality
			\[ \int_{0}^{L} |g^{\prime}(s)|^2 \, ds \leq \liminf_{\varepsilon \to +0} \, E_{\varepsilon}(g^{\varepsilon}) \]
			holds.
	\end{lemma}

\medskip

	Because $D_{j,\varepsilon}^{\delta}$ are tubular domains for fixed $\delta$, we can apply this results.
	So we obtain
	\[ \int_{\delta}^{l_j} |\psi_j^{\prime}(s)|^2 \, ds \leq
			\liminf_{m \to \infty} \, \int_{D_{j,\varepsilon_m}^{\delta}}  |\nabla u^{\varepsilon_m}|^2 \, d\mu_{\varepsilon_m}. \]
	So the following inequality
	\[ \begin{split}
		\sum_{j=1}^N \int_{\delta}^{l_j} |\psi_j^{\prime}(s)|^2 \, ds
		&\leq \liminf_{m \to \infty} \, \sum_{j=1}^N \int_{D_{j,\varepsilon_m}^{\delta}} 
			|\nabla u^{\varepsilon_m}|^2 \, d\mu_{\varepsilon_m} \\[0.1cm]
			&\leq \liminf_{m \to \infty} \, \varphi_{\varepsilon_m}^K(u^{\varepsilon_m})
		\end{split} \]
	holds for $\delta > 0$. By tending $\delta \to +0$, we obtain that
	\begin{equation}\label{e:g-13}
		\varphi^K(\psi) = \sum_{j=1}^N \int_{0}^{l_j} |\psi_j^{\prime}(s)|^2 \, ds
			\leq \liminf_{m \to \infty} \, \varphi_{\varepsilon_m}^K(u^{\varepsilon_m}) 
				= \liminf_{\varepsilon \to +0} \, \varphi_{\varepsilon}^K(u^{\varepsilon}).
	\end{equation}

	On the other hand, it follows that by the transformation of variables $x = \varepsilon_m z$
	\[ \begin{split}
		\varphi_{\varepsilon_m}^V (u^{\varepsilon_m})
		&= \int_{J_{\varepsilon_m}}  \frac{1}{\varepsilon_m} \, V\left( \frac{x}{\varepsilon_m} \right) |u^{\varepsilon_m}(x)|^2 \, 
			\frac{dx}{\omega \varepsilon_m^{n-1}} \\[0.1cm]
		&= \int_{\varepsilon_0^{-1}J}  V(z) |v^{\varepsilon_m}(z)|^2 \, \frac{dz}{\omega}.
		\end{split} \]
	From the same arguments in the proof of Lemma \ref{l:psi-c}, we obtain that
	\begin{equation}\label{e:g-14}
			\lim_{m \to \infty} \, \varphi_{\varepsilon_m}^V (u^{\varepsilon_m})
			= \int_{\varepsilon_0^{-1}J}  V(z) |\psi(O)|^2 \, \frac{dz}{\omega}
			= \varphi^V(\psi).
	\end{equation}
	So (\ref{e:g-13}) and (\ref{e:g-14}) imply (\ref{e:g-11}), that is the condition ($\mathrm{F}1$).

\bigskip

	Next, we check ($\mathrm{F}2$).
	For any $\psi \in D(\varphi) = H^1(G)$,
	we have to make a functional sequence $\{ u^{\varepsilon} \}_{\varepsilon \in I}$ 
	with $u^{\varepsilon} \in L^2(\Omega_{\varepsilon}, d\mu_{\varepsilon})$ such that
	\begin{equation}\label{e:g-21}
		\lim_{\varepsilon \to +0} \, 
			\| u^{\varepsilon} - \psi \circ f_{\varepsilon} \|_{L^2(\Omega_{\varepsilon}, d\mu_{\varepsilon})} = 0, \qquad
		\lim_{\varepsilon \to +0} \, \varphi_{\varepsilon}(u^{\varepsilon}) = \varphi(\psi).
	\end{equation}
	Since $\psi_j \in H^1(0,l_j)$, we can define the functions $\psi_j^{\varepsilon} \in H^1(\varepsilon l,l_j)$ as
	\begin{equation*}
		\psi_j^{\varepsilon}(s) := \psi_j \bigl( \gamma_{\varepsilon}(s) \bigr),
		\qquad \gamma_{\varepsilon}(s) := \frac{l_j}{l_j - \varepsilon l}(s - \varepsilon l) \qquad (\varepsilon l < s < l_j)
	\end{equation*}
	for $j=1,\ldots,N$.
	Moreover we can define the function $u^{\varepsilon} \in C(\overline{\Omega_{\varepsilon}})$ as
	\begin{equation*}
		u^{\varepsilon}(x) := \left\{ \begin{array}{cc}
			\psi_j^{\varepsilon} ( \pi_1 R_j^{-1}x ) & \quad (x \in D_{j,\varepsilon}), \\[0.2cm]
			\psi(O) & \quad (x \in J_{\varepsilon}).
			\end{array} \right. 
	\end{equation*}
	Hereafter we show that $\{ u^{\varepsilon} \}_{\varepsilon \in I}$ satisfies (\ref{e:g-21}) below.

	To prove the first condition of (\ref{e:g-21}), we estimate the following norm
	\begin{equation*}
		\begin{split}
			\| u^{\varepsilon} - \psi \circ f_{\varepsilon} \|_{L^2(\Omega_{\varepsilon}, d\mu_{\varepsilon})}^2
			&= \int_{J_{\varepsilon}}  |u^{\varepsilon} - \psi \circ f_{\varepsilon}|^2 \, d\mu_{\varepsilon}
				+ \sum_{j=1}^N \int_{D_{j,\varepsilon}}  |u^{\varepsilon} - \psi \circ f_{\varepsilon}|^2 \, d\mu_{\varepsilon} \\[0.1cm]
			&= \frac{1}{\omega \varepsilon^{n-1}} \int_{J_{\varepsilon}}  |\psi(O) - (\psi \circ f_{\varepsilon})(x)|^2 \, dx
				+ \sum_{j=1}^N \int_{\varepsilon l}^{l_j} |\psi_j^{\varepsilon}(s) - \psi_j(s)|^2 \, ds \\[0.1cm]
			&\leq \| \psi(O) - \psi \|_{L^{\infty}(G)}^2 \frac{|J_{\varepsilon}|}{\omega \varepsilon^{n-1}}
				+ \sum_{j=1}^N \int_{\varepsilon l}^{l_j} |\psi_j(\gamma_{\varepsilon}(s)) - \psi_j(s)|^2 \, ds.
		\end{split}
	\end{equation*}
	Hence this limit condition
	\[ \lim_{\varepsilon \to +0} \, 
		\| u^{\varepsilon} - \psi \circ f_{\varepsilon} \|_{L^2(\Omega_{\varepsilon}, d\mu_{\varepsilon})} = 0 \]
	holds since (\ref{e:ordJD}) and $\gamma_{\varepsilon}(s) \to s$ as $\varepsilon \to +0$.

\medskip

	To prove the second condition of (\ref{e:g-21}), we use the transformation of variable $t = \gamma_{\varepsilon}(s)$. So we have
	\[ \begin{split}
		\int_{\Omega_{\varepsilon}}  |\nabla u^{\varepsilon}(x)|^2 \, d\mu_{\varepsilon}
		&= \sum_{j=1}^N \int_{\varepsilon l}^{l_j} \left| \frac{d \psi_j^{\varepsilon}}{ds}(s) \right|^2 ds \\[0.1cm]
		&= \sum_{j=1}^N 
			\int_{\varepsilon l}^{l_j} |\psi_j^{\prime}( \gamma_{\varepsilon}(s)) \gamma_{\varepsilon}^{\prime}(s)|^2 \, ds \\[0.1cm]
		&= \sum_{j=1}^N \frac{l_j}{l_j - \varepsilon l} \, \int_{0}^{l_j} |\psi_j^{\prime}(t)|^2 \, dt.
	\end{split} \]
	Therefore $u^{\varepsilon} \in H^1(\Omega_{\varepsilon}) = D(\varphi_{\varepsilon}^K)$ and
	\begin{equation}\label{e:g-25}
		\lim_{\varepsilon \to +0} \, \varphi_{\varepsilon}^K (u^{\varepsilon}) = \sum_{j=1}^N \int_{0}^{l_j} |\psi_j^{\prime}(t)|^2 \, dt
			= \varphi^K(\psi).
	\end{equation}
	Also we obtain that
	\begin{equation}\label{e:g-26}
		\begin{split}
			\varphi_{\varepsilon}^V (u^{\varepsilon})
			&= \int_{J_{\varepsilon}} \frac{1}{\varepsilon} \, 
				V\left( \frac{x}{\varepsilon} \right) |\psi(O)|^2 \, \frac{dx}{\omega \varepsilon^{n-1}} \\[0.1cm]
			&= \frac{|\psi(O)|^2}{\omega} \, \int_{\varepsilon_0^{-1}J}  V(z) \, dz \\[0.1cm]
			&= \varphi^V(\psi)
		\end{split}
	\end{equation}
	for $\varepsilon \in I$.
	Since (\ref{e:g-25}) and (\ref{e:g-26}), the sequence $\{ u^{\varepsilon} \}_{\varepsilon \in I}$ satisfies (\ref{e:g-21}).

\medskip

	By above argument, the proof is completed.
	Lastly, we prove Lemma \ref{l:tube-lsc}. 
	
	\noindent
	(Proof of Lemma \ref{l:tube-lsc})
	
	We define a functional sequence $\{ h^{\varepsilon} \} \subset L^2(Q_1, d\mu_1)$ as 
	$h^{\varepsilon}(y) := g^{\varepsilon}(y_1, \varepsilon y^{\prime})$.
	Then by the direct calculations we have
	\[ \int_{Q_1} |h^{\varepsilon} - g \circ \pi_1 |^2 \, d\mu_1 = 
		\int_{Q_{\varepsilon}} |g^{\varepsilon} - g \circ \pi_1 |^2 \, d\mu_{\varepsilon}
		\ \longrightarrow \ 0 \qquad (\varepsilon \to +0) \]
	and
	\[ E_1(h^{\varepsilon}) = \int_{Q_1} |\nabla h^{\varepsilon}|^2 \, d\mu_1
		= \int_{Q_{\varepsilon}} \left( \left| \frac{\partial g^{\varepsilon}}{\partial x_1} \right|^2 
			+ \varepsilon^2 | \nabla_{x^{\prime}} g^{\varepsilon} |^2 \right) d\mu_{\varepsilon}
		\leq \int_{Q_{\varepsilon}} |\nabla g^{\varepsilon}|^2 \, d\mu_{\varepsilon} = E_{\varepsilon}(g^{\varepsilon}). \]
	Hence $\{ h^{\varepsilon} \}$ is bounded in $H^1(Q_1, d\mu_1)$.
	Therefore we obtain that $g \in H^1(Q_1, d\mu_1)$ and
	\[ \int_{Q_1} |g^{\prime}|^2 \, d\mu_1 \leq \liminf_{\varepsilon \to +0} \, \int_{Q_1} |\nabla h^{\varepsilon}|^2 \, d\mu_1. \]
	Since above inequalities it follows that
	\[ \int_{0}^{L} |g^{\prime}(s)|^2 \, ds 
		= \int_{Q_1} |g^{\prime}|^2 \, d\mu_1 \leq \liminf_{\varepsilon \to +0} \, E_{1}(h^{\varepsilon})
			\leq \liminf_{\varepsilon \to +0} \, E_{\varepsilon}(g^{\varepsilon}). \]
	This is the conclusion of lemma which we omit to prove before.
\end{proof}

\bigskip

\begin{theorem}\label{p:M-c}
	The functional sequence $\{ \varphi_{\varepsilon} \}_{\varepsilon \in I}$ defined by (\ref{d:phie}) Mosco converges to
	the function $\varphi$ defined by (\ref{d:phi}) as $\varepsilon \to +0$ in the sense of Definition \ref{d:M-c}.
\end{theorem}

\begin{proof}
	This statement follows from Proposition \ref{p:G-c} ($\Gamma$-convergence),	Lemma \ref{l:psi-c} (asymptotic compactness)
	and Lemma \ref{l:g-M}.
\end{proof}

\bigskip

Therefore we can apply Theorem \ref{t:K-S} established by Kuwae and Shioya \cite{Kuwae-Shioya}.

\section{Remarks about more general network case}

We consider about simple thin domains, which have a single junction point, on the previous section.
So we remark more general case.

It is not difficult to see that our previous arguments work
(by refining the transformation of variables $\alpha_{\varepsilon}$ in the tubular domains which both side are junction domains)
if a connected graph $G = (V,E)$ satisfies the following conditions:
\begin{enumerate}
\item The set of all edges $E$ is finite.
\item All edges have a finite length.
\end{enumerate}
Because under these assumptions, thin domains $\Omega_{\varepsilon}$ is bounded,
we can apply the compactly embedding theorem in each part.
\begin{figure}[htbp]
		\begin{center}
		\includegraphics[width=50mm]{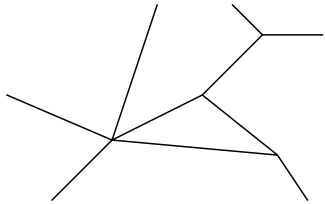}
		\caption{network shaped graph $G$}
		\end{center}
\end{figure}

\section*{Acknowledgment}

The author would like to express my gratitude to S. Albeverio and P. Exner
for insightful comments and suggestions,
and thanks to C. Cacciapuoti for illuminating discussions.


\end{document}